\documentclass[11pt,reqno]{amsart}
\usepackage[T1]{fontenc}
\usepackage[utf8]{inputenc}
\usepackage[a4paper,margin=3.5cm]{geometry}
\usepackage[pagebackref,colorlinks=true,linkcolor=blue,citecolor=blue]{hyperref}

\usepackage{lmodern,mathtools,natbib,todonotes,amsmath,amssymb}
\title[LDP for the largest eigenvalue of the sum of random matrices]
{Large deviations for the largest eigenvalue of the sum of two random matrices}
\thanks{This work was supported by the LABEX MILYON (ANR-10-LABX-0070) of Universit\'e de Lyon  within the program "Investissements d'Avenir" (ANR-11-IDEX- 0007)  and  by the Labex CEMPI (ANR-11-LABX-0007-01) of Universit\'e de Lille
operated by the French National Research Agency (ANR)}
\author{Alice~Guionnet}
\address[Alice Guionnet]{ENS Lyon, France}
\email{aguionne@umpa.ens-lyon.fr}

\author{Mylène~Maïda}%
\address[Mylène Maïda]{Université de Lille, France}%
\email{mylene.maida@univ-lille.fr}%

\email{}

\date{\today}
\numberwithin{equation}{section}
\subjclass[2000]{}  % Set functions and measures on topological spaces 
\newtheorem{theorem}{Theorem}%

\newtheorem{assumption}{Assumption}%
\newtheorem{corollary}[theorem]{Corollary}%
\newtheorem{lemma}[theorem]{Lemma}%
\newtheorem{remark}[theorem]{Remark}%
\newtheorem{proposition}[theorem]{Proposition}%
\newcommand{\dC}{\mathbb{C}}
\newcommand{\dE}{\mathbb{E}}

\newcommand{\dR}{\mathbb{R}}

\newcommand{\ri}{\mathrm{i}}

\newcommand{\dd}{\mathrm{d}} % nice differential
\newcommand{\e}{\mathrm{e}} % nice exponential

\newcommand{\la}{\lambda}

\newcommand{\R}{\dR}

\renewcommand{\Im}{\mathfrak{Im}}

\renewcommand{\d}{ {\mathrm d}}

\newcommand{\eq}{\begin{equation}}
\newcommand{\qe}{\end{equation}}
\renewcommand{\epsilon}{\varepsilon}

\makeatletter

\def\@MRExtract#1 #2!{#1}     % thanks, Martin!
\renewcommand{\MR}[1]{% we need to strip the "(...)"
  \xdef\@MRSTRIP{\@MRExtract#1 !}%
  \href{http://www.ams.org/mathscinet-getitem?mr=\@MRSTRIP}{MR-\@MRSTRIP}}
\makeatother

\makeatletter
\def\namedlabel#1#2{\begingroup
    #2%
    \def\@currentlabel{#2}%
    \phantomsection\label{#1}\endgroup
}
\makeatother

\begin{document}

\maketitle

\begin{abstract}
In this paper, we consider the addition of two matrices in generic position, namely $A+UBU^*,$ where $U$ is drawn under the Haar measure on the unitary or the orthogonal group. We show that, under mild conditions on the empirical spectral measures of the deterministic matrices $A$ and $B,$ the law of the largest eigenvalue satisfies a large deviation principle, in the scale $N,$ with an explicit rate function involving the limit of spherical integrals. We cover in particular all the cases when $A$ and $B$ have no outliers.  
\end{abstract}

\section{Introduction}
\label{sec:intro}
Understanding the spectrum of the sum $A+B$ of two Hermitian matrices knowing the spectra of $A$ and $B$ respectively is a classical and difficult problem. Since the pioneering works of  \cite{Vo91}, we know that free probability provides efficient tools to describe, at least asymptotically, the spectrum of the sum of two large Hermitian matrices in generic position from one another. More precisely, if $A_N$ and $B_N$ are two deterministic $N \times N$ Hermitian matrices and $U_N$ is a unitary random matrix distributed according to the Haar measure, then, in the large $N$ limit, $A_N$ and $U_NB_NU_N^*$ are asymptotically free and the spectral distribution of $H_N:=A_N + U_NB_NU_N^*$ is given by the free convolution of the spectral distributions of $A_N$ and $B_N.$ This global law, that is the convergence of the spectral distribution of $H_N$ at macroscopic scale, has been studied in details by \cite{Sp93,PaVa00} among others. The local law, that is the comparison of the spectral distribution of $H_N$ with the free additive convolution of the spectral distributions of $A_N$ and $B_N$ below the macroscopic scale was then investigated by \cite{Ka12} and \cite{BaErSc17}. 
In this paper, we will be interested in the behavior of the largest eigenvalue of $H_N.$ As a corollary of the results of \cite{CoMa14} on strong asymptotic freeness, we know that if $A_N$ and $B_N$ have no outliers, then the largest eigenvalue of $H_N$ converges to the right edge of the support of the free convolution of the spectral distributions of $A_N$ and $B_N.$  %No general result is known about the fluctuations of this largest eigenvalue. 
In this work, we investigate the large deviations of this extreme eigenvalue.

In the framework of random matrix theory, there are very  few large deviation results known about the spectrum, basically because the eigenvalues are complicated functions of the entries. A notable exception is given by the Gaussian invariant ensembles
for which the joint law of the eigenvalues can be explicitly written as a Coulomb gas. Based on this explicit formula, 
large deviation principles for the spectral measure at global scale have been established by \cite{BAGu97} and for the  largest eigenvalue by \cite{BAdeGu01}. Another special case is given by the sum of a deterministic matrix and a Gaussian invariant ensemble. Then, the spectrum can be constructed as the realization at time one of a Hermitian (or symmetric) Brownian motion starting from a given deterministic matrix. This point of view was used by \cite{GuZe02} to study the large deviations of the empirical measure, and the large deviations for the process of the largest eigenvalue starting from the origin were derived by \cite{DoMa12}. One of the application of this paper is to provide the large deviation for the largest eigenvalue of this sum by using another approach based on spherical integrals.   Beyond these cases where specific tools are available,  it was observed by \cite{BoCa14} that deviations of the spectrum of  Wigner matrices for which the distribution of the entries has a tail which is heavier than Gaussian are naturally created by big entries. This key remark allowed to obtain the large deviations for the empirical measure in  \citep{BoCa14} (see also \citep{Gr17} for the counterpart for covariance matrices) and for the largest eigenvalue in \citep{Au16}.  Large deviations for the spectrum of Wigner matrices with sub-gaussian entries is still completely open as far as the empirical measure is concerned.  One can  mention the deviations results of \cite{Au} for the moments of the spectral measure in several models. Concerning the deviations of the largest eigenvalue, beyond the works \citep{BAdeGu01, DoMa12, Au16} already cited above, the following models have been so far studied~: Gaussian ensembles plus a rank one perturbation by \cite{Ma07}, very thin covariance matrices by \cite{FeHoKl08}, finite rank perturbations of deterministic matrices or unitarily invariant ensembles by \cite{BGGuMa12}. In a companion paper, \cite{GuHu18} have established a large deviation principle for the largest eigenvalue of Wigner matrices with entries having sharp sub-Gaussian tails, such as Rademacher matrices.  They show that the speed and the rate function of this large deviation principle are the same as in the Gaussian case. %In a forthcoming  paper \citep{AuGuHu} with F. Augeri, they extend these results in the  general sub-Gaussian tails.

{\bf Acknowledgments} The idea to  tilt measures by the spherical integral came out magically from a discussion with M. Potters in UCLA in 2017 and we wish to thank him for this beautiful inspiration. We also benefited from many discussions with J. Husson and F. Augeri with whom one of the author   is working on a companion project on Wigner matrices.  Finally, we are very grateful for stimulating discussions with O. Zeitouni and N. Cook.

\section{Statement of the results}
\label{sec:results}

Let $(A_N)_{N \ge 1}$ and $(B_N)_{N \ge 1}$ be two sequences of deterministic real diagonal matrices, with $A_N$ and $B_N$ of size $N \times N.$ We denote by $\la_1^{(A_N)} \ge \ldots \ge \la_N^{(A_N)}$ and $\la_1^{(B_N)} \ge \ldots \ge \la_N^{(B_N)}$ 
their respective eigenvalues in decreasing order, by
$$\|A_N\| := \max (| \la_1^{(A_N)}|,| \la_N^{(A_N)}|)  \textrm{ and }  \|B_N\| := \max (| \la_1^{(B_N)}|,| \la_N^{(B_N)}|)$$ their respective spectral radius and by $$\hat \mu_{A_N} := \frac{1}{N} \sum_{j=1}^N \delta_{\la_j^{(A_N)}}
 \textrm{ and } \hat \mu_{B_N} := \frac{1}{N} \sum_{j=1}^N \delta_{\la_j^{(B_N)}}$$ their respective spectral measures.

For $\beta = 1$ or $2,$ we denote by $m_N^\beta$ the Haar measure on the orthogonal group $\mathcal O_N$ if $\beta=1$ and on the unitary group $\mathcal U_N$ if $\beta =2.$ For any $U$ a $N \times N$ matrix, we denote by 
$H_N(U):= A_N + UB_NU^*$ and by $\la_{\rm max}^N$ the largest eigenvalue of $H_N(U).$ The goal of the present work is to establish a large deviation principle for the law of $\la_{\rm max}^N$ under the Haar measure $m_N^\beta.$
This large deviation principle holds under mild  assumptions  that we now detail. 

\begin{assumption} \label{hyp}\mbox{}\\
\begin{itemize}
 \item[\namedlabel{hyp:esd}{($H_{\rm bulk}$)}] The sequences of spectral empirical measures $(\hat \mu_{A_N})_{N \ge 1}$ and   $(\hat \mu_{B_N})_{N \ge 1}$ converge weakly as $N$ grows to infinity respectively  to ${\mu_a}$ and ${\mu_b}$, compactly supported on $\dR$. Moreover, $\sup_{N \ge 1} (\|A_N\| + \|B_N\|) <\infty.$
%   \item[\namedlabel{hyp:conc}{($C_\beta$)}] There exists  $\varepsilon >0$ and a sequence of deterministic probability measures $(\nu_N)_{N \ge 1}$ such that, 
%   $$ \limsup_{N \rightarrow \infty}\frac{1}{N} \log m_N^\beta \left( \dd(\hat \mu_N, \nu_N) > N^{-\varepsilon}\right) =-\infty$$  
   \item[\namedlabel{hyp:extreme}{($H_{\rm edge}$)}] The largest eigenvalues  $\la_1^{(A_N)}$ and $\la_1^{(B_N)}$   converge  as $N$ grows to infinity to $\rho_a$ and $\rho_b$ respectively. \\
\end{itemize}
\end{assumption}

%[\textsf{[Esd]}] [\textsf{[Freeconvo]}][ \textsf{[Extreme]}]
% As it is often the case, a crucial mechanism for the deviations of the largest eigenvalue is the fact that the spectral measure concentrates around its limit in a faster scale 
% than the deviations of the largest eigenvalue. This is expressed by the assumption \ref{hyp:conc}: in Appendix \ref{sec:appendix}, we will give some sufficient conditions for the assumption   \ref{hyp:conc} to hold.
A key argument of the proof will be a tilt of the measure by a rank one  spherical integral. Similar strategies are used in the companion paper \citep{GuHu18} % and \citep{AuGuHu} 
to study some classes of sub-Gaussian Wigner matrices. The rank one  spherical integral is defined as follows:  for any $\theta \ge 0$ and $M_N$ an Hermitian matrix of size $N,$ 
$$ I_N^\beta (\theta, M_N) := \int \e^{N\theta(UM_NU^*)_{11}}m_N^\beta(\d U)\quad \textrm{ and } \quad J_N^\beta (\theta, M_N) := \frac{1}{N} \log I_N^\beta (\theta, M_N). $$ 
The rate function of our large deviation principle  will crucially  involve the limit of $J_N^\beta (\theta, H_N)$ as $N$ grows to infinity, which we now describe. For  $\mu$ a compactly supported probability measure on $\R,$ we denote by  
$ \mathsf r(\mu)$ the right edge of the support of $\mu$ and by $G_\mu$ the Stieltjes transform of $\mu$~: for $\la \ge \mathsf r(\mu),$    
$$ G_\mu(\la) := \int \frac{1}{\la-y} \mu(\d y).$$
It is decreasing on the interval $( \mathsf r(\mu), \infty).$ By taking the limit as $\lambda$ decreases to $ \mathsf r(\mu),$ one can also define $G_\mu(\mathsf r(\mu)) \in \R_+ \cup \infty.$ As $ G_\mu$ is bijective
from $( \mathsf r(\mu), \infty)$ to $(0,G_\mu(\mathsf r(\mu)) ),$ one can define its inverse on this latter interval, that we denote by $K_\mu.$ Then, for any  $z \in (0,G_\mu(\mathsf r(\mu)) ),$ we define 
$$ R_\mu(z) := K_\mu(z) - \frac{1}{z}.$$
The function $R_\mu$ is called the $R$-transform fo $\mu.$ One can check that $R_\mu$ is increasing and that 
$\lim_{z \rightarrow 0} R_\mu(z) = \int \la \mu(\dd \la),$ so that it is bijective from $(0, G_\mu(\mathsf{r}(\mu)))$ to 
$\left( \int \la \mu(\dd \la), \mathsf{r}(\mu)) - \frac{1}{G_\mu(\mathsf{r}(\mu))}\right).$ We denote by $Q_\mu$ its inverse on this interval. 
We can now define, for $\beta = 1$ or $2,$ $\theta \ge 0,$ $\mu$ a compactly supported probability measure and $\rho \ge \mathsf{r}(\mu)$:
\[
J_\mu^{\beta}(\theta,\rho):= \left\{\begin{array}{ll}
 \frac{\beta}{2} \int_0^{\frac{2\theta}{\beta}}  R_\mu(u) \d u, & \textrm{if } 0 \le \frac{2\theta}{\beta} \le G_\mu(\rho),\\
\theta \rho -  \frac{\beta}{2}\log \theta -  \frac{\beta}{2} \int \log(\rho-y) \mu(\d y) +  \frac{\beta}{2}\left(\log  \frac{\beta}{2} -1\right), & \textrm{if } \frac{2\theta }{\beta}>  G_\mu(\rho).
\end{array}
\right.
% \theta v_\mu^\beta(\theta,  \rho) - \frac{\beta}{2}\int \log\left(1+ \frac{2}{\beta}\theta v_\mu^\beta(\theta,  \la) - \frac{2}{\beta} \theta y\right) \mu(\d y),
% \qe
% with 
% \eq
%  v_\mu^\beta(\theta,\la) := \left\{\begin{array}{ll}
%                                  R_\mu\left(\frac{2\theta}{\beta}\right), & \textrm{if } 0 \le \frac{2\theta}{\beta} \le G_\mu(\rho),\\
%                                  \rho - \frac{\beta}{2\theta}, & \textrm{if } \theta >  G_\mu(\rho).
%                                 \end{array}
% \right.
\]
% One can check that, if $0 \le \frac{2\theta}{\beta} \le G_\mu(\rho)$, we have
% $$ J_\mu^{\beta}(\theta,\rho):= \frac{\beta}{2} \int_0^{\frac{2\theta}{\beta}}  R_\mu(u) \d u.$$
% By Theorem 6 in \citep{GuMa05}, we know that under Assumptions \eqref{hyp:esd} and \eqref{hyp:extreme},
% \eq
% \lim_{N \rightarrow \infty} J_N^\beta (\theta, A_N) = J_{\mu_a}^{\beta}(\theta,\rho_a)
% \qe
% are similarly for $B_N.$
If $\mu_1$ and $\mu_2$ are two probability measures compactly supported on $\dR,$ we denote by $\mu_1 \boxplus \mu_2$ the free convolution of  $\mu_1$ and $\mu_2.$ It is uniquely determined  as the unique probability measure with $R$-transform equal to the sum of the R-transforms of $\mu_1$ and $\mu_2$ (see  \citep{Vo91}). For any $\theta \ge 0$ and  $x \ge \mathsf r({\mu_a} \boxplus{\mu_b}),$ we denote by 
\[
I^\beta(\theta, x) := J^\beta_{{\mu_a} \boxplus{\mu_b}}(\theta,  x) - J^\beta_{\mu_a}(\theta, \rho_{a}) - J^\beta_{\mu_b}(\theta, \rho_{b}),
\]
and 
\eq  \label{def:Ib}
I^\beta(x) :=  \left\{\begin{array}{ll}
                 \sup_{\theta \ge 0} I^\beta(\theta, x), & \textrm{ if } x \ge \mathsf r({\mu_a} \boxplus{\mu_b}),\\
                 +\infty, & \textrm{ otherwise.}
                \end{array}
                \right.
\qe
It is easy to check the following:
\begin{lemma}
 \label{lem:propI}
Let $\mu_a$, $\mu_b$, $\rho_a$ and $\rho_b$ be given as in Assumption \ref{hyp}.  For $\beta = 1$ or $2,$ the function $I^\beta$ is a good rate function. Moreover, for any $x > \rho_a+ \rho_b,$ $ I^\beta(x) =+\infty.$  
\end{lemma}

The proof will be given at the beginning of Section \ref{sec:propI}. We can now state the main results of this paper.  The first result is   the following large deviation upper bound:
\begin{proposition}\label{prop:up}
 Under Assumption \ref{hyp}, for $\beta = 1$ or $2,$  for any $x \in \R,$
$$ \limsup_{\delta\downarrow 0} \limsup_{N \rightarrow +\infty} \frac{1}{N}\log m_N^\beta\left( \la_{\rm max}^N \in [x-\delta, x+\delta]\right) \le - I^\beta(x).$$
\end{proposition}
We will then derive 
 the following large deviation lower bound:

\begin{proposition}\label{prop:down}
Assume that Assumption \ref{hyp} holds and that $\mu_a$ is not a Dirac mass at $\rho_a$ and $\mu_b$ is not a Dirac mass at $\rho_b.$   
Then, for $\beta = 1$ or $2,$   for any $x \in \R$ such that
\eq \label{hyp:outlier}
G_{{\mu_a} \boxplus{\mu_b}}(x) \le \min \left(G_{\mu_a}(\rho_a), G_{\mu_b}(\rho_b)\right),
 \qe
 we have
$$ \liminf_{\delta\downarrow 0} \liminf_{N \rightarrow +\infty} \frac{1}{N}\log m_N^\beta\left( \la_{\rm max}^N \in [x-\delta, x+\delta]\right) \ge - I^\beta(x).$$
\end{proposition}
This leads to the following important corollary:
\begin{theorem}\label{theo:main}
 Under Assumption \ref{hyp} and if moreover,
\eq \label{hyp:colle}
G_{{\mu_a} \boxplus{\mu_b}}(\mathsf{r}({\mu_a} \boxplus{\mu_b})) \le \min \left(G_{\mu_a}(\rho_a), G_{\mu_b}(\rho_b)\right), \tag{\textsf{NoOut}}
 \qe
 then, for $\beta = 1$ or $2,$  the law of $ \la_{\rm max}^N $ under $m_N^\beta$ satisfies a large deviation principle in the scale $N$ with good rate function $I^\beta.$
\end{theorem}
One can in fact check (see Lemma \ref{lem:colle} for more details) that the condition \eqref{hyp:colle} is automatically satisfied if there is no outliers, namely $\rho_a= \mathsf{r}(\mu_a)$ and $\rho_b= \mathsf{r}(\mu_b).$ This leads to the following corollary
\begin{corollary}\label{cor:colle}
 Under the assumption \ref{hyp:esd}, if $A_N$ and $B_N$ have no outliers, then  for $\beta = 1$ or $2,$  the law of $ \la_{\rm max}^N $ under $m_N^\beta$ satisfies a large deviation principle in the scale $N$ with good rate function $I^\beta.$ 
\end{corollary}

Observe that in the case where one of the measures $\mu_a$ or $\mu_b$ is a Dirac mass at $\rho_a$ or $\rho_b$ respectively and the other matrix has no outliers,  $\mathsf{r}(\mu_a\boxplus\mu_b)=\rho_a+\rho_b$ so that the above result still holds, but with a degenerate rate function which is infinite except at $\rho_a+\rho_b$.
To get a taste of what happens in the case with outliers, we also consider in Appendix \ref{sec:appendix}
 the following model:
let $(U^{(1)}, \ldots, U^{(d)})$ be independent random matrices with distribution $m_N^\beta,$ independent of $U$ and $\gamma_1, \ldots, \gamma_d$ be nonnegative real numbers. For any $1 \le i \le d,$ we denote by $U^{(i)}_1$ the first column vector of $U^{(i)}$ and we set:
\begin{equation}
\label{def:deformed}
X_N:= A_N + UB_NU^* + \sum_{i=1}^d \gamma_i U^{(i)}_1(U^{(i)}_1)^*. 
\end{equation}
We show in Theorem \ref{theo:deformed} that we still have a large deviation principle, for which the rate function will depend on the $\gamma_i$'s.
The rest of the paper will be organized as follows: in the next section, we will first prove a more general result than Proposition \ref{prop:up}, that holds not only for $m_N^\beta$ but also for a whole family of tilted measures.
This will be helpful in the proof of Proposition \ref{prop:down}, that will be developed in Section \ref{sec:down}. Before getting there, we will study  in Section \ref{sec:propI} some properties of the rate function $I^\beta.$
The last section will be devoted to the proof of Theorem \ref{theo:main} and Corollary \ref{cor:colle}, with
Lemma \ref{lem:colle} as prerequisite.  
 At the end of the paper, in Appendix \ref{sec:appendix}, we will study the deviations of the largest eigenvalue of $X_N$ for the deformed model \eqref{def:deformed}.

%This work was supported by the LABEX MILYON (ANR-10-LABX-0070) of Universit\'e de Lyon, within the program "Investissements d'Avenir" (ANR-11-IDEX- 0007) operated by the French National Research Agency (ANR).

\section{Large deviation upper bound for tilted measures}
\label{sec:up}
For $\theta \ge 0,$ $\beta = 1$ or $2,$ we define a tilted measure on $\mathcal O_N$ if $\beta=1$ and $\mathcal U_N$ if $\beta =2$ as follows
$$ m_N^{\beta,\theta}(\d U) := \frac{I_N^\beta(\theta, A_N + UB_NU^*)}{I_N^\beta(\theta, A_N)I_N^\beta(\theta, B_N)} m_N^\beta(\d U).$$
It is easy to check that $m_N^{\beta,\theta}$ is a probability measure: indeed, for any $U,$ we have that $I_N^\beta(\theta, A_N + UB_NU^*)\ge 0$ and 
$\mathbb E_{m_N^\beta}(I_N^\beta(\theta, A_N + UB_NU^*)) = I_N^\beta(\theta, A_N)I_N^\beta(\theta, B_N).$ For these tilted measures, we have the following weak large deviation upper bound :
\begin{proposition}
 \label{tiltedub}
Under Assumption \ref{hyp}, for $\beta = 1$ or $2,$  for any $\theta \ge 0,$  
for any $x < r({\mu_a} \boxplus{\mu_b}),$
\eq  \label{tiltedub2}
 \limsup_{\delta\downarrow 0} \limsup_{N \rightarrow +\infty} \frac{1}{N}\log m_N^{\beta,\theta}\left(  \la_{\rm max}^N \in [x-\delta, x+\delta]\right) = - \infty,
 \qe 
and   for any $x \ge r({\mu_a} \boxplus{\mu_b}),$
\eq  \label{tiltedub1}
 \limsup_{\delta\downarrow 0} \limsup_{N \rightarrow +\infty} \frac{1}{N}\log m_N^{\beta,\theta}\left( \la_{\rm max}^N \in [x-\delta, x+\delta]\right) \le -\left[ I^\beta(x) - I^\beta(\theta,x)\right].
 \qe
\end{proposition}

\begin{remark}
 Applying this proposition with $\theta = 0$ gives Proposition \ref{prop:up}.
\end{remark}

 As we will see in Section \ref{sec:down}, establishing an upper bound for any $\theta \ge 0$ will be useful in the proof of Proposition \ref{prop:down}. 
To prove Proposition \ref{tiltedub}, and in particular its first statement, we will need to check that, under $m_N^{\beta,\theta}$ the spectral measure $$\hat \mu_N := \frac{1}{N} \sum_{j=1}^N \delta_{\la_j^{(H_N(U))}}$$ of $H_N(U) =A_N + U B_N U^*$ concentrates around a  deterministic probability measure $\nu_N^\beta$ much faster than $\e^{-N}.$ A natural choice for this deterministic equivalent of $\hat \mu_N$ will be its expectation $\dE_{m_N^\beta}\hat\mu_N.$ 
More precisely,  we equip the set $\mathcal P(\R)$ of probability measures on $\R$  with the bounded Lipschitz distance $\dd$: for any Lipschitz function $f : \R \to \R,$ we define $\|f\|_\infty:= \sup_{x \in \R} |f(x)|$ and $\|f\|_{\rm Lip}:= \sup_{x \neq y} \frac{|f(x)-f(y)|}{|x-y|},$ then for any $\mu$ and $\nu$ in $\mathcal P(\R),$
 $$ \dd(\mu,\nu) := \sup_{\substack{ \|f\|_\infty\le 1 \\ \|f\|_{\rm Lip}\le 1}} \int f \dd\mu - \int f \dd\nu.$$
We then have the following concentration result:

\begin{lemma}
 \label{lem:conc}
Under Assumption \ref{hyp:esd}, for $\beta=1$ or $2$ and any $\theta \ge 0,$ 
 $$\limsup_{N \rightarrow \infty} \frac{1}{N}\log m_N^{\beta,\theta}\left(\dd(\hat \mu_{N}, \dE_{m_N^\beta}\hat\mu_N) > N^{-1/4}\right) = - \infty .$$
\end{lemma}

\begin{proof}
Let $\beta=1$ or $2$ and  $\theta \ge 0$ be fixed. For any Borel subset $A$ of $\mathcal O_N$ if $\beta =1$ and   $\mathcal U_N$ if $\beta =2,$ we have:
\begin{align*}
  m_N^{\beta,\theta}(A) & =  \frac{1}{I_N^\beta(\theta, A_N)I_N^\beta(\theta, B_N)} \int_A I_N^\beta(\theta, A_N + UB_NU^*) m_N^\beta(\d U) \\
   & \le \e^{2 N\theta K }\, m_N^\beta(A),
\end{align*}
with $K:=  \sup_{N \ge 1}(\|A_N\| +\|B_N\|),$ which is assumed to be finite. Therefore it is enough to prove Lemma \ref{lem:conc} for $\theta =0,$ that is
 $$\limsup_{N \rightarrow \infty} \frac{1}{N}\log m_N^{\beta}\left(\dd(\hat \mu_{N}, \dE_{m_N^\beta}\hat\mu_N) > N^{-1/4}\right) = - \infty .$$
For $\beta = 2,$ Theorem 3.8 in \citep{MeMe13} states that there exists $c, C >0$ such that 
\eq \label{meckes}
m_N^{2}\left(\dd(\hat \mu_{N}, \dE_{m_N^2}\hat\mu_N) > N^{-1/4}\right) \le C\e^{-cN^{3/2}},
\qe
from which the lemma follows.
A careful reading of \citep{MeMe13}  shows that the exact same result as \eqref{meckes} also holds for $\beta  =1.$
% For any 1-Lipschitz function $f : \R \to \R$, the function defined on $\mathcal O_N$ if $\beta=1$ and $\mathcal U_N$ if $\beta=2$ by $U \mapsto \int f \dd \hat \mu_N$ is $\frac{1}{\sqrt N}$-Lipschitz (see e.g. \cite{AGZ}{Lemma 2.3.1}).
% Therefore, according e.g. to \cite{GrMi83}{p. 128},  there exists a constant $c_\beta >0$ such that, for any $\delta >0,$
% $$  m_N^{\beta}\left( \left|\int f \dd \hat \mu_N - \dE_{m_N^\beta} \int f \hat\mu_N\right|>\delta\right) \le \e^{- c_\beta N^2 \delta^2 }.$$
% By Assumption \ref{hyp:esd}, we have $C=\sup_{N \ge 1}(\|A_N\| +\|B_N\|) <\infty$ and therefore, for any $U$ and $N,$ $\hat \mu_{N}$ and $\dE_{m_N^\beta} \int f \hat\mu_N$ are supported in $[-K,K]$ and 
% $$ \dd(\hat \mu_{N}, \dE_{m_N^\beta}\hat\mu_N) = \sup_{f\in \mathsf{Lip}_K}\left|\int f \dd \hat \mu_N - \dE_{m_N^\beta} \int f \hat\mu_N\right|,$$
% where $\mathsf{Lip}_K$ is the set of Lipschitz functions supported on $[-K,K],$ such that $\|f\|_\infty\le 1 $ and $ \|f\|_{\rm Lip}\le 1.$ 
% Now, mimicking for example the proof of Theorem 1.3 a) in \cite{GuZe00}, there exists a family of $\frac{4K}{\delta}$ 1-bounded 1-Lipschitz functions which can approximate any function in  $\mathsf{Lip}_K$ at the price of an error of at most $\frac{\delta}{4}.$ This leads to the following upper bound:
% $$  m_N^{\beta}\left( \sup_{f \in \mathsf{Lip}_K}\left|\int f \dd \hat \mu_N - \dE_{m_N^\beta} \int f \hat\mu_N\right|> \delta \right) \le \frac{4K}{\delta} \e^{-c_\beta N^2 \left(\frac{\delta^2}{16K}\right)^2},$$
% Choosing $\delta = N^{-1/8},$ this concludes the proof. 
\end{proof}
 We can now prove Proposition \ref{tiltedub}.
In the sequel, we will denote by $\nu_N^\beta:= \dE_{m_N^\beta}\hat\mu_N.$ 

\begin{proof}[Proof of Proposition \ref{tiltedub}]
The first claim \eqref{tiltedub2} is a direct consequence of the previous lemma. 
 Indeed, let $x < \mathsf r(\mu_a \boxplus \mu_b)$ and $\delta_0 := \frac{ \mathsf r(\mu_a \boxplus \mu_b)-x}{2}.$ Then, for any $\delta \le \delta_0,$ there exists
$ \varepsilon(\delta) >0,$
\eq \label{upupup}
\{ \la_{\rm max}^N \in [x-\delta, x+\delta]\} \subset \{\dd(\hat \mu_{N},\mu_{a} \boxplus  \mu_{b}) > \varepsilon(\delta)\}.
\qe
Using Corollary 5.4.11 for $\beta =2$ and Exercise 5.4.18 for $\beta =1$ in \citep{AGZ}, we know that $\nu_N^\beta$ converges weakly to $\mu_a \boxplus \mu_b$ as $N$ goes to infinity. As the distance $\dd$ metrizes the weak convergence,   for $N$ large enough,
$$ \{ \la_{\rm max}^N \in [x-\delta, x+\delta]\} \subset \{\dd(\hat \mu_{N},\nu_N^\beta) > \varepsilon(\delta)/2\}$$
so that, by Lemma \ref{lem:conc}, for any $\delta \le \delta_0,$
 $$\limsup_{N \rightarrow \infty} \frac{1}{N}\log m_N^{\beta,\theta}\left(  \la_{\rm max}^N \in [x-\delta, x+\delta]\right) = - \infty.$$
 
We now prove \eqref{tiltedub1}.  Let $\delta >0$ and $x \ge \mathsf r(\mu_a \boxplus \mu_b)$ be fixed and define the following event:
\eq \label{def:event}
 \mathsf E_{N,\delta}^x := \left\{  \la_{\rm max}^N \in [x-\delta, x+\delta], \dd(\hat \mu_{N}, \nu_N^\beta) \le N^{-1/4}\right\}.
 \qe
Then we have,
$$ m_N^{\beta,\theta}\left(  \la_{\rm max}^N \in [x-\delta, x+\delta]\right)  \le m_N^{\beta,\theta}( \mathsf E_{N,\delta}^x) +m_N^{\beta,\theta}( \dd(\hat \mu_{N}, \nu_N^\beta) > N^{-1/4}). $$
By Lemma \ref{lem:conc}, it is therefore enough to show that
$$\limsup_{\delta\downarrow 0} \limsup_{N \rightarrow \infty} \frac{1}{N}\log m_N^{\beta,\theta}\left( \mathsf E_{N,\delta}^x \right) \le -\left[ I^\beta(x) - I^\beta(\theta,x)\right].$$
To lighten a bit the notations we write $A, B$ and $H$  for $A_N,$ $B_N$ and $H_N = A_N + UB_NU^*$ respectively. For any $\theta, \theta^\prime \ge 0,$ we have
\begin{align*}
 m_N^{\beta,\theta}( \mathsf E_{N,\delta}^x) & =  \frac{1}{I_N^\beta(\theta, A)I_N^\beta(\theta, B)} \mathbb E_{m_N^\beta} \left( \mathsf 1_{ \mathsf E_{N,\delta}^x} I_N^\beta(\theta,H)\frac{ I_N^\beta(\theta^\prime,H)}{ I_N^\beta(\theta^\prime,H)}\right) \\
 & \le  \frac{ \mathbb E_{m_N^\beta} ( I_N^\beta(\theta^\prime,H)) }{I_N^\beta(\theta, A)I_N^\beta(\theta, B)} \sup_{U \in  \mathsf E_{N,\delta}^x} \frac{I_N^\beta(\theta,A+UBU^*)}{ I_N^\beta(\theta^\prime,A+UBU^*)} \\
 & =  \frac{ I_N^\beta(\theta^\prime, A)I_N^\beta(\theta^\prime, B)}{I_N^\beta(\theta, A)I_N^\beta(\theta, B)} \sup_{U \in  \mathsf E_{N,\delta}^x} \frac{I_N^\beta(\theta,A+UBU^*)}{ I_N^\beta(\theta^\prime,A+UBU^*)} \
\end{align*}
We now have to estimate $ \sup_{U \in  \mathsf E_{N,\delta}^x} I_N^\beta(\theta,A+UBU^*)$:  we will use the continuity of spherical integrals derived in \citep{Ma07} that states as follows. Let $(G_N)_{N \ge 1}$ a sequence of deterministic matrices such that $\sup_{N \ge 1} \|G_N\|<\infty$ and for any $N \ge 1,$ $\la_1^{(G_N)}=x$ and $ \d(\hat \mu_{G_N}, \nu_N^\beta) \le N^{-1/4}.$ According to Proposition 2.1 in \citep{Ma07}, for any $\theta \ge 0,$ there exists a continuous function $g_\theta$ such that $g_\theta(0) = 0$ and for any $U \in  \mathsf E_{N,\delta}^x,$
\[ \left|\frac{1}{N} \log I_N^\beta(\theta, A+UBU^*) - \frac{1}{N} \log I_N^\beta(\theta, G_N)\right| \le g_\theta(\delta).\]
Therefore,
\begin{align*}
 \limsup_{N \rightarrow \infty}  \frac{1}{N} \log   m_N^{\beta,\theta}( \mathsf E_{N,\delta}^x) & \le \lim_{N \rightarrow \infty} ( J_N^\beta(\theta^\prime, A) +  J_N^\beta(\theta^\prime, B) - J_N^\beta(\theta, A) - J_N^\beta(\theta,B))\\
 &  + \lim_{N \rightarrow \infty} ( J_N^\beta(\theta, G_N) - J_N^\beta(\theta^\prime, G_N)) +    g_\theta(\delta)+ g_{\theta^\prime}(\delta),\\
 & \le  - (I^\beta(\theta^\prime,x) - I^\beta(\theta,x)) +  g_\theta(\delta)+ g_{\theta^\prime}(\delta),
\end{align*}
where at the last line, we have used Theorem 6 in \citep{GuMa05}. Letting $\delta$ going to zero and then optimizing over $\theta^\prime \ge 0,$ we get the required upper bound.
\end{proof}

\section{Properties of the rate function $I^{\beta}$}
\label{sec:propI}

We now check the properties of the rate  function $I^\beta$ defined in \eqref{def:Ib}.

\begin{proof}[Proof of  Lemma \ref{lem:propI}]
 An ingredient for the proof if the following: for any compactly supported $\mu,$ for any $\theta \ge 0$ and $ \rho  \ge \mathsf{r}(\mu)$ such that $\theta \le G_\mu(\rho),$  
we have 
\eq \label{eq:encadrement}
\rho-\frac{1}{\theta} \le R_\mu(\theta) \le \rho- \frac{1}{ G_\mu(\rho)} .
\qe
Indeed, as $K_\mu$ is a decreasing function, we have 
$ R_\mu(\theta) = K_\mu (\theta) -\frac{1}{\theta} \ge  \rho-\frac{1}{\theta}.$ On the other hand, the limit of $R_\mu(\theta)$ as $\theta $ grows to $G_\mu(\rho)$ is $\rho- \frac{1}{ G_\mu(\rho)}.$ 
As $R_\mu$ is nondecreasing, we get the upper bound.
Moreover, it is easy to check that, for any $x \ge 0,$ there exists $C, C^\prime \in \R$ (depending on $\mu$ and $x$ but not on $\theta$) such that, for $\theta$ large enough, we have
$$ \theta x  - \frac{\beta}{2} \log\theta + C \le J_\mu^\beta(\theta, x) \le  \theta x +C^\prime,$$
so that, for any $x \ge 0,$ there exists $c, c^\prime \in \R$  such that, for $\theta$ large enough,
$$ \theta (x -\rho_a - \rho_b)  - \frac{\beta}{2} \log\theta +c \le I^\beta(\theta, x) \le  \theta (x -\rho_a - \rho_b)  + \beta\log\theta +c^\prime.$$
If $x> \rho_a + \rho_b,$ letting $\theta$ grow to infinity, we obtain that $I^\beta( x) = +\infty.$

If $\theta \ge 0$ is small enough, 
$$ I^\beta(\theta, x) = \frac{\beta}{2} \int_0^{\frac{2\theta}{\beta}} (R_{\mu_a \boxplus \mu_b}(u) - R_{\mu_a}(u) - R_{\mu_b}(u)) \dd u =0,$$
by the properties of the $R$-transform. The function $I^\beta$ is therefore nonnegative. 
If we denote by $g$ the lower semi-continuous function which is equal to $-\infty$ on $[\mathsf r(\mu_a \boxplus \mu_b), +\infty)$   and $+\infty$ outside, then 
$I^\beta = \sup (g, \sup_{\theta} I^\beta(\theta, \cdot))$   is lower semi-continuous as a supremum of lower semi-continuous functions.
As it is infinite outside the interval $[\mathsf{r}(\mu_a \boxplus \mu_b), \rho_a + \rho_b],$ it is a good rate function.
\end{proof}

To perform the tilt leading to   the lower bound, we will need to further study the  properties of the function $I^\beta.$ 
% Assume that $\mathsf r(\mu_a \boxplus \mu_b) \le x < \rho_a+\rho_b$ is such that 
%  $$ G_{\mu_a\boxplus \mu_b}(x) \le \min(G_{\mu_a}(\rho_a), G_{\mu_b}(\rho_b)).$$
%  One can assume without loss of generality that $G_{\mu_a}(\rho_a) \le G_{\mu_b}(\rho_b)$ (note that these quantities can be infinite).
%  Then the function  $I^\beta(\theta, x)$ is 

\begin{lemma} \label{uniqueness}
 Under Assumption \ref{hyp},  for any $\mathsf r(\mu_a \boxplus \mu_b) \le x < \rho_a+\rho_b$ such that 
 $$ G_{\mu_a\boxplus \mu_b}(x) \le \min(G_{\mu_a}(\rho_a), G_{\mu_b}(\rho_b)),  $$
then, for $\beta= 1$ or $2,$ there exists a unique $\theta \ge 0$  such that 
 $$ I^\beta(\theta, x) = \sup_{\theta^\prime \ge 0} I^\beta(\theta^\prime,x).$$
 We denote  by $\theta_x^\beta := {\rm argmax}_{\theta \ge 0} I^\beta(\theta,x).$ For  any $\mathsf r(\mu_a \boxplus \mu_b) \le x < \rho_a+\rho_b$ 
 and $\mathsf r(\mu_a \boxplus \mu_b) \le y \le \rho_a+\rho_b$ 
such that $x \neq y,$
$$ \sup_{\theta \ge 0} I^\beta(\theta,y) > I^\beta(\theta_x^\beta,y).$$
\end{lemma}

% The rest of this section will be devoted to the proof of Proposition \ref{uniqueness}.

\begin{proof}[Proof of Lemma \ref{uniqueness}]
Let  $\mathsf r(\mu_a\boxplus \mu_b) \le x <\rho_a+\rho_b$ such that
 $$ G_{\mu_a\boxplus \mu_b}(x) \le \min(G_{\mu_a}(\rho_a), G_{\mu_b}(\rho_b)).$$
 
The first remark is that  if  $G_{\mu_a}(\rho_a)$ and $ G_{\mu_b}(\rho_b) $ are infinite, then  
$\mathsf r(\mu_a \boxplus \mu_b) \ge \rho_a+\rho_b$ 
and  there is nothing to check. 
Indeed, if $G_{\mu_a}(\rho_a) = G_{\mu_b}(\rho_b) =\infty,$ we see by the inequalities \eqref{eq:encadrement}, that
$$ \lim_{x \rightarrow \infty} R_{\mu_a}(x) = \rho_a \quad \textrm{ and } \quad \lim_{x \rightarrow \infty} R_{\mu_b}(x) = \rho_b, $$
 so that 
 $$ \lim_{x \rightarrow \infty} K_{\mu_a \boxplus \mu_b}(x) = \rho_a + \rho_b  \quad \textrm{ and } \quad  \lim_{x \rightarrow \rho_a + \rho_b  } G_{\mu_a \boxplus \mu_b}(x) = \infty,$$
 leading to $\mathsf r(\mu_a\boxplus \mu_b) \ge \rho_a+\rho_b.$
By symmetry of the problem, without loss of generality, one can now assume that $G_{\mu_a}(\rho_a) \le G_{\mu_b}(\rho_b)$ and  $G_{\mu_a}(\rho_a) <\infty.$

With the function $I^\beta$ defined in \eqref{def:Ib}, if we denote by $I^\beta_x$ the function $\theta \mapsto I^\beta(\theta,x),$
 then there exist some constants $C_1, C_2$ and $C_3$ (that may depend on $\mu_a, \rho_a, \mu_b, \rho_b$ and $x$ but not on $\theta$) such that
 
\begin{equation}\label{formuleI}
I_x^\beta(\theta) = \left\{ \begin{array}{ll}
                     0, &\textrm{ if } 0 \le \frac{2\theta}{\beta} \le G_{\mu_a\boxplus \mu_b}(x), \\
                     \theta x -\frac{\beta}{2} \log \theta -\frac{\beta}{2} \int_0^{\frac{2\theta}{\beta}} (R_{\mu_a}+R_{\mu_b})(u)\d u  + C_1, &\textrm{ if }   G_{\mu_a\boxplus \mu_b}(x)  \le \frac{2\theta}{\beta} \le G_{\mu_a}(\rho_a),\\
                     \theta(x-\rho_a) - \frac{\beta}{2} \int_0^{\frac{2\theta}{\beta}} R_{\mu_b}(u)\d u + C_2, &\textrm{ if }  G_{\mu_a}(\rho_a) \le {\frac{2\theta}{\beta}} \le G_{\mu_b}(\rho_b),\\
                     \theta(x-\rho_a-\rho_b) +  \frac{\beta}{2} \log \theta + C_3,& \textrm{ if }\frac{2\theta}{\beta} \ge G_{\mu_b}(\rho_b),
                    \end{array}
\right.\nonumber
\end{equation}
where the last line does not occur if $ G_{\mu_b}(\rho_b)= \infty.$ 
In the computation, we have used the well known fact that $R_{\mu_a \boxplus \mu_b} = R_{\mu_a} +R_{\mu_b}$ when the three functions are well defined. Therefore, one can check that the function  $I^\beta_x$  is continuously differentiable and its derivative is given by:
\[
(I_x^\beta)^\prime(\theta) = \left\{ \begin{array}{ll}
                     0, &\textrm{ if } 0 \le \frac{2\theta}{\beta} \le G_{\mu_a\boxplus \mu_b}(x),\\
                      x - K_{\mu_a\boxplus \mu_b}\left( \frac{2\theta}{\beta}\right), &\textrm{ if }    G_{\mu_a\boxplus \mu_b}(x)  \le \frac{2\theta}{\beta} \le G_{\mu_a}(\rho_a),\\
                     x-\rho_a - R_{\mu_b}\left( \frac{2\theta}{\beta}\right), &\textrm{ if }   G_{\mu_a}(\rho_a) \le {\frac{2\theta}{\beta}} \le G_{\mu_b}(\rho_b),\\
                     x-\rho_a-\rho_b +\frac{\beta}{2\theta},& \textrm{ if }\frac{2\theta}{\beta} \ge G_{\mu_b}(\rho_b).
                    \end{array}
\right.
\]
We now set $\alpha_x:= \frac{1}{\rho_a+\rho_b-x}.$ We claim  that $$\alpha_x \ge  G_{\mu_a}(\rho_a).$$ Indeed, $K_{\mu_b}$ is well defined on the interval $(0,G_{\mu_b}(\rho_b) )$, so that 
$K_{ \mu_b}( G_{\mu_a}(\rho_a))$ and therefore $K_{\mu_a\boxplus \mu_b}( G_{\mu_a}(\rho_a))$ are well defined. As $K_{\mu_a\boxplus \mu_b}$ is a decreasing function, we have:
$$ G_{\mu_a\boxplus \mu_b}(x) \le  G_{\mu_a}(\rho_a)  $$
and this implies:
$$ x  \le  K_{\mu_a\boxplus \mu_b}( G_{\mu_a}(\rho_a)) =  K_{\mu_a}( G_{\mu_a}(\rho_a)) + K_{ \mu_b}( G_{\mu_a}(\rho_a)) - \frac{1}{G_{\mu_a}(\rho_a)}$$
As $K_{\mu_b}$ is also  a decreasing function, this yields:
$$ x \le  K_{\mu_a}( G_{\mu_a}(\rho_a)) + K_{ \mu_b}( G_{\mu_b}(\rho_b)) - \frac{1}{G_{\mu_a}(\rho_a)} = \rho_a + \rho_b  - \frac{1}{G_{\mu_a}(\rho_a)}, $$
which is equivalent to $\alpha_x \ge  G_{\mu_a}(\rho_a).$ There are therefore two cases to consider and we claim that:
\begin{itemize}
 \item[\textsf{Case 1:}] If $ G_{\mu_a}(\rho_a) \le \alpha_x <  G_{\mu_b}(\rho_b),$ then $I^\beta_x$ reaches its maximum at $$\theta_x^\beta:= \frac{\beta}{2} R_{\mu_b}^{(-1)}(x-\rho_a);$$
 \item[\textsf{Case 2:}] if $\alpha_x \ge  G_{\mu_b}(\rho_b),$ then $I^\beta_x$ reaches its maximum at $\theta_x^\beta:= \frac{\beta}{2} \alpha_x.$
\end{itemize}
%In both cases, there is a unique point $\theta_x^\beta$ where the maximum is reached and Lemma \ref{uniqueness} holds. 
Let us now prove this claim. On the interval $\left[0, \frac{\beta}{2} G_{\mu_a}(\rho_a)\right],$ the function $(I^\beta_x)^\prime$ is nondecreasing and it vanishes  at zero, it is therefore nonnegative so that
$I^\beta_x$ is nondecreasing on this interval. 
We have
$$ (I^\beta_x)^\prime\left(\frac{\beta}{2} G_{\mu_a}(\rho_a)\right) \ge 0 \quad\textrm{ and } \quad (I^\beta_x)^\prime\left(\frac{\beta}{2} G_{\mu_b}(\rho_b)\right) = -\frac{1}{\alpha_x} + \frac{1}{G_{\mu_b}(\rho_b)}.$$ 
Moreover, as $R_{\mu_b}$ is an increasing function, $(I^\beta_x)^\prime$ is decreasing on the interval $\left[\frac{\beta}{2} G_{\mu_a}(\rho_a), \frac{\beta}{2} G_{\mu_b}(\rho_b)\right].$
We now distinguish the two cases. 

In \textsf{Case 1},  $(I^\beta_x)^\prime\left(\frac{\beta}{2} G_{\mu_b}(\rho_b)\right)<0,$ and therefore there exists $$\theta_x \in \left[\frac{\beta}{2} G_{\mu_a}(\rho_a), \frac{\beta}{2} G_{\mu_b}(\rho_b)\right)$$ such that $I^\beta_x$ is increasing on $\left[\frac{\beta}{2} G_{\mu_a}(\rho_a),\theta_x\right]$ and then decreasing. One can check that the point where  $(I^\beta_x)^\prime$ cancels is given by $\frac{\beta}{2} R_{\mu_b}^{(-1)}(x-\rho_a).$ Moreover,  $(I^\beta_x)^\prime$ is decreasing on $ \left[\frac{\beta}{2} G_{\mu_b}(\rho_b),\infty\right)$ and negative at  $\frac{\beta}{2} G_{\mu_b}(\rho_b)$ so it remains negative and $I^\beta_x$ is decreasing on this interval. The first claim holds true.

In \textsf{Case 2},  $(I^\beta_x)^\prime\left(\frac{\beta}{2} G_{\mu_b}(\rho_b)\right)\ge 0,$ and therefore $I^\beta_x$ is increasing on the interval $\left[\frac{\beta}{2} G_{\mu_a}(\rho_a), \frac{\beta}{2} G_{\mu_b}(\rho_b)\right].$ But $(I^\beta_x)^\prime$
is nonnegative at $\frac{\beta}{2} G_{\mu_b}(\rho_b),$ decreasing on $ \left[\frac{\beta}{2} G_{\mu_b}(\rho_b),\infty\right)$ and converges to $x-\rho_a-\rho_b <0$ as $\theta$ grows to $\infty.$   Therefore,  there exists $\theta_x \in \left(\frac{\beta}{2} G_{\mu_b}(\rho_b), \infty\right)$ such that $I^\beta_x$ is increasing on $\left(\frac{\beta}{2} G_{\mu_b}(\rho_b),\theta_x\right]$ and then decreasing. One can check that the point where  $(I^\beta_x)^\prime$ cancels is given by $\frac{\beta}{2} \alpha_x$ and the second claim holds true. This concludes the proof of the uniqueness of $\theta.$

Moreover, looking carefully at the definition of $\theta_x^\beta$ in \textsf{Case 1} and \textsf{Case 2}, one can see that it is an increasing function of $x.$ In particular, for $x \neq y$ such that 
$\mathsf r(\mu_a \boxplus \mu_b) \le x,y < \rho_a+\rho_b,$ $\theta_x^\beta \neq \theta_y^\beta $ and  therefore $ \sup_{\theta \ge 0} I^\beta(\theta,y) > I^\beta(\theta_x^\beta,y).$

We now have to deal with the case when $y = \rho_a+\rho_b,$ that is to show that:
\eq \label{eq:rhoab}
  \sup_{\theta \ge 0} I^\beta(\theta,\rho_a+\rho_b) > I^\beta(\theta_x^\beta,\rho_a+\rho_b).
  \qe
  If $G_{\mu_b}(\rho_b)$ is finite, for $\theta > \frac{\beta}{2} G_{\mu_b}(\rho_b),$ 
  $$I^\beta(\theta,\rho_a+\rho_b)=\frac{\beta}{2}\log \theta+C_3$$
  and therefore
  the supremum is infinite and \eqref{eq:rhoab} holds. Otherwise let us first consider the case where
  $\mu_b = \delta_{\rho_b}.$ We claim that in this case, the condition $\mathsf r(\mu_a \boxplus \mu_b) \le x < \rho_a+\rho_b$ and $ G_{\mu_a\boxplus \mu_b}(x) \le \min(G_{\mu_a}(\rho_a), G_{\mu_b}(\rho_b))$  are never simultaneously satisfied. Indeed, in this case, $\mu_a \boxplus \mu_b$ is just a shift of $\mu_a$ by $\rho_b,$ so that, for any $x< \rho_a+\rho_b,$
 $ G_{\mu_a\boxplus \mu_b}(x) = G_{\mu_a}(x-\rho_b) > G_{\mu_a}(\rho_a),$ as $G_{\mu_a}$ is decreasing. 
  If $\mu_b \neq \delta_{\rho_b},$ then, there exists $\alpha \in (0,1]$ and $M$ finite such that, for any $x \ge \rho_b,$ 
 $$  G_{\mu_b}(x)  \le \frac{1-\alpha}{x-\rho_b} +M.$$
 From there, we get that, for any $u > G_{\mu_a}(\rho_a)\vee \frac{2M}{\alpha},$
 $$ u \le \frac{1-\alpha}{K_{\mu_b}(u)-\rho_b}+M \quad \textrm{ so that } \quad R_{\mu_b}(u) \le \rho_b - \frac \alpha {2u}.$$
 Therefore, there exist $c, c^\prime \in \R,$ such that for any $\theta \ge  G_{\mu_a}(\rho_a)\vee \frac{2M}{\alpha},$ 
 $$ I^\beta(\theta,\rho_a+\rho_b) \ge  \theta\rho_b -\frac{\beta}{2} \int_{\frac{2M}{\alpha}}^{\frac{2\theta}{\beta} }\left( \rho_b - \frac \alpha {2u}\right)du +c=
 \frac {\beta\alpha}{4 } \log \theta + c^\prime$$ 
 so that, letting $\theta$ grow to infinity, we get again that $ I^\beta(\rho_a+\rho_b) = \infty$  and   \eqref{eq:rhoab} holds.
This concludes the proof of Lemma \ref{uniqueness}.
\end{proof}

\section{Large deviation lower bound}
\label{sec:down}

The goal of this section is to show Proposition \ref{prop:down}. A classical strategy to get a large deviation lower bound is to tilt the measure in such a way that the rare event $\{\la_{\rm max}^N \in [x-\delta, x+\delta]\}$ becomes typical under the tilted measure. 
We now check that it is possible to make such a tilt:
\begin{lemma}\label{tiltedLGN}
Under Assumption \ref{hyp}, for any  $x\in [\mathsf r({\mu_a} \boxplus {\mu_b}), \rho_a+\rho_b)$ such that 
  $$ G_{\mu_a\boxplus \mu_b}(x) \le \min(G_{\mu_a}(\rho_a), G_{\mu_b}(\rho_b)),  $$
for $\beta= 1$ or $2,$ we have  
 $$ \lim_{\delta\downarrow 0} \liminf_{N \rightarrow \infty} \frac{1}{N}\log m_N^{\beta,\theta_x^\beta}\left( \mathsf E_{N,\delta}^x\right) \ge 0,$$
 where $\mathsf E_{N,\delta}^x$ was defined in \eqref{def:event} and $\theta_x^\beta$ in Lemma \ref{uniqueness}.
 \end{lemma}

 \begin{proof}[Proof of Lemma \ref{tiltedLGN}] 
Let  $\beta = 1$ or $2$ and $\mathsf r({\mu_a} \boxplus {\mu_b}) \le x < \rho_a+\rho_b$ be fixed. Let $y \neq x$ be such that $y <\mathsf r({\mu_a} \boxplus {\mu_b})$ or $y >  \rho_a+\rho_b.$ By Lemma \ref{lem:propI}, we know that $I^\beta(y) =\infty$, so that, by  Proposition \ref{tiltedub}, we have
$$ \lim_{\delta\downarrow 0} \limsup_{N \rightarrow \infty} \frac{1}{N}\log m_N^{\beta,\theta_x^\beta}\left(  \la_{\rm max}^N\in [y-\delta, y+\delta]\right) = - \infty.$$
Let now $y \neq x$ be such that $\mathsf r({\mu_a} \boxplus {\mu_b}) \le y \le \rho_a+\rho_b.$ Then, by Proposition \ref{tiltedub} ,  we have
\begin{align*}
 \lim_{\delta\downarrow 0} \limsup_{N \rightarrow \infty} \frac{1}{N}\log m_N^{\beta,\theta_x^\beta}\left( \la_{\rm max}^N\in [y-\delta, y+\delta]\right) & 
  \le -( \sup_{\theta \ge 0} I^\beta(\theta,y) -  I^\beta(\theta_x^\beta,y))  
\end{align*}
% Moreover,
% \begin{multline*}
%  \lim_{\delta\downarrow 0} \limsup_{N \rightarrow \infty} \frac{1}{N}\log m_N^{\beta,\theta_x^\beta}\left( \la_{\rm max}^N\in [\rho_a+\rho_b, \rho_a+\rho_b+\delta]\right) \\ 
%   \le -( \sup_{\theta^\prime \ge 0} I^\beta(\theta^\prime, \rho_a+\rho_b) -  I^\beta(\theta_x^\beta,\rho_a+\rho_b)) 
% \end{multline*}
As a consequence, if we denote by 
\[ L_x^\beta(y):= \left\{
\begin{array}{ll}
 \sup_{\theta \ge 0} I^\beta(\theta,y) -  I^\beta(\theta_x^\beta,y), & \textrm{if } \mathsf r({\mu_a} \boxplus {\mu_b}) \le x \le \rho_a+\rho_b,\\
\infty, & \textrm{otherwise,}
 \end{array}
\right.
\]
we know that the law of $\la_{\rm max}^N$ under $m_N^{\beta,\theta_x^\beta}$ satisfies a weak large deviation upper bound with good rate function $ L_x^\beta.$
Moreover, for $N$ large enough, $\la_{\rm max}^N$ lies with probability one in the compact set $[ \mathsf r({\mu_a} \boxplus {\mu_b})-1, \rho_a+\rho_b+1],$ so that  it is in fact a large deviation upper bound. By 
 Lemma \ref{uniqueness}, we know that $L_x^\beta$ is nonnegative and  vanishes only at $x$. Therefore, we deduce that, for any $\delta >0,$ for $N$ large enough,
 $$ m_N^{\beta,\theta_x^\beta}\left( \la_{\rm max}^N\in [x-\delta, x+\delta]\right) \ge \frac{3}{4}\,.$$
 But, in virtue of Lemma \ref{lem:conc}, for $N$ large enough, we also have 
 $$m_N^{\beta,\theta_x^\beta}\left(\dd(\hat \mu_{N}, \nu_N^\beta) \le N^{-1/4} \right) \ge \frac{3}{4} $$
 so that $$ m_N^{\beta,\theta_x^\beta}\left( \mathsf E_{N,\delta}^x\right) \ge \frac{1}{2},$$
 and Lemma \ref{tiltedLGN} follows.
 \end{proof}

From there, one can easily get the large deviation lower bound.

\begin{proof}[Proof of Proposition \ref{prop:down}] Let $\beta = 1$ or $2$ and $x \ge \mathsf r({\mu_a} \boxplus {\mu_b})$ be fixed.
If $x > \rho_a+\rho_b$ or $x<\mathsf r({\mu_a} \boxplus {\mu_b}),$ Lemma \ref{lem:propI} gives that $I^\beta(x) = \infty,$ so that the lower bound obviously holds. 
Moreover, as we have seen at the end of the proof of Lemma \ref{uniqueness}, as $\mu_b$ is not a Dirac mass at $\rho_b$,  then $ I^\beta(\rho_a+\rho_b) = \infty$ and the lower bound also holds for $x = \rho_a+\rho_b$.

Let us now assume that 
$\mathsf r({\mu_a} \boxplus {\mu_b}) \le x < \rho_a+\rho_b$
and let $\theta_x^\beta$ be the corresponding shift defined in Lemma \ref{uniqueness}. Then, with $\mathsf E_{N,\delta}^x$ defined in \eqref{def:event}, we have:
\begin{align*}
m_N^\beta(\la_{\rm max}^N \in [x-\delta,x+\delta]) & \ge  m_N^\beta(\mathsf E_{N,\delta}^x)   = \dE_{m_N^\beta} \left( \mathsf{1}_{\mathsf E_{N,\delta}^x} \frac{I_N^\beta(\theta_x^\beta, H)}{I_N^\beta(\theta_x^\beta,H)}\right) \\
 & \ge  \inf_{U\in \mathsf E_{N,\delta}^x} \frac{1}{I_N^\beta(\theta_x^\beta, A+UBU^*)} \\
  & \omit \hfill $\times I_N^\beta(\theta_x^\beta, A)  I_N^\beta(\theta_x^\beta, B)m_N^{\beta,\theta_x^\beta}(\mathsf E_{N,\delta}^x)$ 
\end{align*} 
so that, using again Proposition 2.1 in \citep{Ma07}, we get: 
\begin{align*}
 \liminf_{N \rightarrow \infty} \frac{1}{N}\log m_N^\beta\left( \la_{\rm max}^N  \in [x-\delta, x+\delta]\right) &\ge - I^\beta(\theta_x^\beta,x) - g_{\theta_x^\beta}(\delta) \\
 & +  \liminf_{N \rightarrow \infty} \frac{1}{N}\log m_N^{\beta,\theta_x^\beta}\left( \mathsf E_{N,\delta}^x\right).
\end{align*}
Letting $\delta$ going to zero and using Lemma \ref{tiltedLGN}, we get that 
$$ \lim_{\delta\downarrow 0} \liminf_{N \rightarrow \infty} \frac{1}{N}m_N^\beta(\la_{\rm max}^N \in [x-\delta,x+\delta]) \ge - I^\beta(\theta_x^\beta,x) \ge - I^\beta(x).$$

% To finish the proof, we now check that the lower bound holds for $x = \rho_a+\rho_b.$ We distinguish sevral cases.
% 
% As we have seen at the beginning of the proof of Lemma \ref{uniqueness}, if $G_{\mu_a}(\rho_a) = G_{\mu_b}(\rho_b) =\infty,$ then $\mathsf{r}({\mu_a} \boxplus {\mu_b})= \rho_a + \rho_b.$
% In particular, $ \la_{\rm max}^N$ converges almost surely to  $\rho_a + \rho_b,$ so that 
% $$ \liminf_{\delta \rightarrow 0} \liminf_{N \rightarrow \infty} m_N^\beta(\la_{\rm max}^N \in [\rho_a+\rho_b- \delta,\rho_a+\rho_b+\delta]) =0\ge - I^\beta (\rho_a+\rho_b).$$
This concludes the proof.

\end{proof}

\section{Proof of the main theorem and its corollary}
\label{sec:cor}

\begin{proof}[Proof of Theorem \ref{theo:main}] 
Assume that Assumption \ref{hyp} and the condition \eqref{hyp:colle} are satisfied. If we denote by $K:= \sup_{n \ge 1}(\|A_N\|+\|B_N\|),$ which is assumed to be finite, we have that for any $N\ge 1,$
$$ m_N^\beta(\la_{\rm max}^N >2K) =0,$$
so that the exponential tightness is obviously satisfied. By \cite[Lemma 4.1.23]{DeZe98}, it is therefore enough to show a weak large deviation principle. The upper bound is given by Proposition \ref{prop:up} for $\theta=0.$

As for the lower bound, we distinguish three cases, if $ G_{\mu_a}(\rho_a) =  G_{\mu_b}(\rho_b) =\infty,$ as we have seen if the proof of Lemma \ref{uniqueness}, we have that 
$\mathsf r (\mu_a \boxplus \mu_b) = \rho_a+\rho_b.$ In particular, $\lambda^N_{\rm max}$ converges almost surely to $\rho_a+\rho_b,$ so that the lower bound holds.
If $\mu_b = \delta_{\rho_b},$ then $\mu_a \boxplus \mu_b$ is just a shift of $\mu_a$ by $\rho_b,$ so that $\mathsf r (\mu_a  \boxplus \mu_b) = \mathsf r (\mu_a) + \rho_b$
and $G_{\mu_a \boxplus \mu_b}(\mathsf r (\mu_a  \boxplus \mu_b))= G_{\mu_a}(\mathsf r (\mu_a )).$
Assume that $G_{\mu_a}(\rho_a)<\infty.$ If $\mathsf r (\mu_a) <\rho_a,$ then the condition  \eqref{hyp:colle} is not satisfied, because $G_{\mu_a}$ is a decreasing function.
If $\mathsf r (\mu_a) =\rho_a,$ then we have a similar situation as in the previous case, $\lambda^N_{\rm max}$ converges almost surely to $\rho_a+\rho_b,$ so that the lower bound holds.
By symmetry, the same holds true if $\mu_a = \delta_{\rho_a}.$ 
Otherwise and if the condition \eqref{hyp:colle} holds, as $G_{\mu_a \boxplus \mu_b}$ is decreasing, then for any $x \ge \mathsf{r}(\mu_a \boxplus \mu_b),$ we have
$$ G_{\mu_a \boxplus \mu_b}(x) \le \min\left( G_{\mu_a}(\rho_a),  G_{\mu_b}(\rho_b)\right).$$ The lower bound is given by Proposition \ref{prop:down}.
\end{proof}

We now prove Corollary \ref{cor:colle}. Our goal is to show that if $A_N$ and $B_N$ have no outliers, then the condition \eqref{hyp:colle} is automatically satisfied.
Indeed, if $A_N$ and $B_N$ have no outliers, it means  that their respective largest eigenvalues converge to the edge of the support of the limiting measure, that is to say $\rho_a = \mathsf r(\mu_a)$ and  $\rho_b = \mathsf r(\mu_b).$
Therefore, Corollary \ref{cor:colle} is a direct consequence of the following lemma:
\begin{lemma}
\label{lem:colle}
For any probability measures $\mu$ and $\nu$ compactly supported on $\dR,$  we have
$$  G_{\mu\boxplus \nu}(\mathsf r(\mu \boxplus \nu)) \le \min(G_{\mu}(\mathsf r(\mu)), G_{\nu}(\mathsf r(\nu))).$$
\end{lemma}

\begin{proof}
If one of the measures $\mu$ or $\nu$ is a single point mass, the additive free convolution is just a translation and we have equality. We now assume that none of them is a single point mass.
In general, we know (see e.g. \citep{Be08}) that there exists a function $\omega,$ called the subordination function, which is analytic on $\dC^+:=\{z \in \dC, \Im\,z >0\}$ such that, for all $z \in \dC^+,$
\eq \label{subordination}
G_{\mu\boxplus \nu}(z) = G_\mu(\omega(z))
\qe
By \cite[Theorem 2.3]{Be06}, as  $\mu$ or $\nu$ are not a single point mass, $ G_{\mu\boxplus \nu}$ can be continuously extended to $\dC^+ \cup \dR$ with values in $\overline \dC := \dC\cup \infty.$ Moreover, as   $\mu$ and $\nu$ are compactly supported, by \cite[Theorem 3.3(3)]{Be08}, $\omega$ can also be continuously extended to $\dC^+ \cup \dR.$
From \eqref{subordination}, we have that, for any $z \in \dC^+ \cup \dR,$
$$ \Im\,   G_{\mu\boxplus \nu}(z) = - \Im\, \omega(z). \int \frac{\d \mu(t)}{|t-\omega(z)|^2}.$$
Let $z$ be a real number in the interval $(\mathsf r(\mu \boxplus \nu), \infty).$ Then $\int \frac{\d \mu(t)}{|t-\omega(z)|^2} >0$ and  $\Im \,  G_{\mu\boxplus \nu}(z)= 0,$ 
so that  $ \Im \,\omega(z)=0.$ Therefore, $\omega$ restricted to the interval  $(\mathsf r(\mu \boxplus \nu), \infty)$ takes values in $\dR \cup \infty.$ Moreover $\omega(z)$ goes to $\infty$ as $z$ goes to $\infty,$
so that $\omega((\mathsf r(\mu \boxplus \nu), \infty))$ is an interval $I_\omega$ containing a neighborhood of $\infty.$

Let  $a<\mathsf r(\mu)$ such that $(a, \infty) \subset I_\omega.$ For any $y >0,$ we have
$$ - \int_a^{\mathsf r(\mu)}\Im \, G_\mu(x+ \ri y) = \int_a^{\mathsf r(\mu)} \dd \mu(t) \left(\arctan\left(\frac{r(\mu)-t}{y}\right) -  \arctan\left(\frac{a-t}{y}\right)\right).$$

As $y$ decreases to zero, the right hand-side converges to $\pi \mu((a, \mathsf r(\mu))) >0.$ 
On the other hand, for any $x \in (a, \mathsf r(\mu)) \subset \omega((\mathsf r(\mu \boxplus \nu), \infty)),$ there exists $x^\prime > \mathsf r(\mu \boxplus \nu),$ such that $x=\omega(x^\prime)$ and 
$$ \Im\,   G_{\mu}(x) = \Im\,   G_{\mu}(\omega(x^\prime)) = \Im\,  G_{\mu\boxplus \nu}(x^\prime)  =0.$$
As $G_\mu$ is continuous on $\dC^+ \cup \dR,$ by dominated convergence, we get that the left hand-side goes to zero, as $y$ decreases to zero. This leads to a contradiction and we deduce that $I_\omega \subset [\mathsf{r}(\mu), \infty),$
which means, by continuity of $\omega,$ that 
$$ \omega(\mathsf r(\mu \boxplus \nu)) \ge \mathsf r(\mu).$$
As $G_\mu$ is decreasing on $(\mathsf r(\mu),\infty),$ this gives 
$$ G_{\mu\boxplus \nu}(\mathsf r(\mu \boxplus \nu))=G_\mu(\omega(\mathsf r(\mu \boxplus \nu))) \le G_\mu(\mathsf r(\mu))\,.$$
As $\mu$ and $\nu$ play symmetric roles, this concludes the proof of Lemma \ref{lem:colle}.
\end{proof}

\appendix
\section{Study of the deformed model \eqref{def:deformed}}
\label{sec:appendix}

In order to study the deviations of the largest eigenvalue of the deformed model below its expected value, we will need a counterpart of Theorem \ref{theo:main} for the smallest eigenvalue of $H_N.$ We first state the counterpart of the condition \eqref{hyp:colle}. \\

\begin{itemize}
   \item[\namedlabel{hyp:min}{(\textsf{NoDown})}] \textit{ The smallest eigenvalues  $\la_N^{(A_N)}$ and $\la_N^{(B_N)}$   converge  as $N$ grows to infinity to $\ell_a$ and $\ell_b$ respectively and 
   $ \displaystyle
G_{{\mu_a} \boxplus{\mu_b}}(\mathsf{l}({\mu_a} \boxplus{\mu_b})) \ge \max \left(G_{\mu_a}(\ell_a), G_{\mu_b}(\ell_b)\right).$ %\tag{\textsf{NoOut}}
} \\
\end{itemize}

As in Lemma \ref{lem:colle}, one can check that this condition is satisfied if $A_N$ and $B_N$ have no outliers.
We now extend the definition of the rate function $I^\beta$ introduced in \eqref{def:Ib}. For any compactly supported probability measure $\mu,$ we  denote by $\mathsf{l}(\mu)$ the left edge of the support of $\mu.$
For $\beta = 1$ or $2,$ $\theta \le 0,$ $\mu$ a compactly supported probability measure and $\ell \le \mathsf{l}(\mu)$:
\[
J_\mu^{\beta}(\theta,\ell):= \left\{\begin{array}{ll}
 \frac{\beta}{2} \int_0^{\frac{2\theta}{\beta}}  R_\mu(u) \d u, & \textrm{if } G_\mu(\ell) \le \frac{2\theta}{\beta}  \le 0,\\
\theta \ell -  \frac{\beta}{2}\log (-\theta) -  \frac{\beta}{2} \int \log(y-\ell) \mu(\d y) +  \frac{\beta}{2}\left(\log  \frac{\beta}{2} -1\right), & \textrm{if } \frac{2\theta }{\beta}
<G_\mu(\ell).
\end{array}
\right.
\]
For any $\theta \le 0$ and  $x \le \mathsf l({\mu_a} \boxplus{\mu_b}),$ we denote by 
\[
I^\beta(\theta, x) := J^\beta_{{\mu_a} \boxplus{\mu_b}}(\theta,  x) - J^\beta_{\mu_a}(\theta, \ell_{a}) - J^\beta_{\mu_b}(\theta, \ell_{b}),
\]
and 
\eq  \label{def:Ibneg}
I^\beta_{\rm min}(x) :=  \left\{\begin{array}{ll}
                 \sup_{\theta \le 0} I^\beta(\theta, x), & \textrm{ if } x \le \mathsf l({\mu_a} \boxplus{\mu_b}),\\
                 \infty, & \textrm{ otherwise.}
                \end{array}
                \right.
\qe

Applying Theorem \ref{theo:main}  to $-A_N$ and $-B_N,$ one can get a large deviation principle for the smallest eigenvalue $\la_{\rm min}^N$
of $H_N :$
\begin{corollary} \label{cor:min}
Under the assumptions \ref{hyp:esd} and \ref{hyp:min},  for $\beta = 1$ or $2,$  the law of $ \la_{\rm min}^N $ under $m_N^\beta$ satisfies a large deviation principle in the scale $N$ with good rate function $I^\beta_{\rm min}.$
\end{corollary}

For the sake of simplicity, when treating the deformed model, we will stick to the case $\beta =1.$
For any $x > \mathsf{r}(\mu_a \boxplus \mu_b),$ we denote by $\mu_x$ the measure defined as follows: for any bounded measurable function $f,$
$$ \int f(\la) \mu_x(\dd \la) = \int f\left(\frac{1}{x-\la}\right) \mu_a \boxplus \mu_b(\dd\la).$$
If $x = \mathsf{r}(\mu_a \boxplus \mu_b),$ we set 
$$ \int f(\la) \mu_x(\dd \la) = \lim_{y \downarrow x}\int f\left(\frac{1}{y-\la}\right) \mu_a \boxplus \mu_b(\dd\la),$$
whenever it exists.
In particular, for any $x \ge \mathsf{r}(\mu_a \boxplus \mu_b),$ $\int \la \mu_x(\dd \la) = G_{ \mu_a \boxplus \mu_b}(x).$

For any $x \ge \rho \ge \mathsf{r}(\mu_a \boxplus \mu_b)$ and $\ell \le \mathsf{l}(\mu_a \boxplus \mu_b)$ we define 
$$ \alpha_+(\rho) :=  \frac{ G_{ \mu_a \boxplus \mu_b}(\rho)}{1+ (x-\rho) G_{ \mu_a \boxplus \mu_b}(\rho)}
\quad \textrm{ and }\quad  \alpha_-(\ell) :=\frac{ G_{ \mu_a \boxplus \mu_b}(\ell)}{1+ (x-\ell) G_{ \mu_a \boxplus \mu_b}(\ell)}.$$
For $\alpha \in \left(\frac{1}{x-\ell},\frac{1}{x-\rho}\right)$  and $\kappa  \notin \left(\frac{1}{x-\ell},\frac{1}{x-\rho}\right),$ we set
\[
h_{\alpha,x}(\kappa) := \int \log\left(\frac{\kappa - \la}{\kappa - \alpha}\right)\mu_x(\dd \la).
\]
We finally set 
\begin{equation}
\label{def:T+}
T_{x,\rho}^+(\alpha):=\left\{\begin{array}{ll}
                   h_{\alpha,x}(K_{\mu_x}(Q_{\mu_x}(\alpha))), & \textrm{if } \alpha \in [ G_{\mu_a \boxplus \mu_b}(x), \alpha_+(\rho)],\\
                   h_{\alpha,x}\left(\frac{1}{x-\rho}\right), & \textrm{if } \alpha \in \left(\alpha_+(\rho), \frac{1}{x-\rho}\right),\\
                    \infty,  & \textrm{if } \alpha > \frac{1}{x-\rho},
                  \end{array}
\right.
\end{equation}
and
\begin{equation}
\label{def:T-}
T_{x,\ell}^-(\alpha):=\left\{\begin{array}{ll}
                   h_{\alpha,x}(K_{\mu_x}(Q_{\mu_x}(\alpha))), & \textrm{if } \alpha \in [ \alpha_-(\ell), G_{\mu_a \boxplus \mu_b}(x)],\\
                    h_{\alpha,x}\left(\frac{1}{x-\ell}\right), & \textrm{if } \alpha \in \left(\frac{1}{x-\ell}, \alpha_-(\ell)\right)\,\\
                    \infty,  & \textrm{if } \alpha < \frac{1}{x-\ell}.
                  \end{array}
\right.
\end{equation}

Before proving Theorem \ref{theo:deformed}, we need to state a variant of Proposition 16 in \citep{GuMa05}.
Let $(\la_i)_{i \in \mathbb N^*}$ be a sequence of real numbers such that $\frac{1}{N} \sum_{i=1}^N \delta_{\la_i}$ converges to $\mu_a \boxplus \mu_b.$ We denote by $P$ the standard Gaussian measure on $\R$ and we assume that  $(g_1, \ldots, g_N)$ follows the law $P^{\otimes N}.$ 
For any $x \notin \{\la_i, i \in \mathbb N^*\},$ we denote by $v_N(x) = \frac{\sum_{i=1}^N \frac{1}{x-\la_i} g_i^2}{ \sum_{i=1}^N g_i^2}.$
\begin{proposition}
 \label{prop:16}
Assume that $\max_{i=1}^N \la_i$ converges, as $N$ grows to $\infty,$ to $\rho \ge \mathsf{r}(\mu_a \boxplus \mu_b).$ Then, for any $x \ge \rho$ and $\alpha \in \R$ such that $\alpha \ge G_{\mu_a \boxplus \mu_b}(x),$ we have 
 \[ \lim_{\delta \downarrow 0} \lim_{N \rightarrow \infty} \frac{1}{N}\log P^{\otimes N}\left( v_N(x) \in \left[\alpha-\delta, \alpha+\delta\right]\right)
 =- T_{x,\rho}^+\left(\alpha\right).
 \]
 Assume that $\min_{i=1}^N \la_i$ converges, as $N$ grows to $\infty,$ to $\ell \le \mathsf{l}(\mu_a \boxplus \mu_b).$ Then, for any $x \ge \mathsf{r}(\mu_a \boxplus \mu_b)$ and $\alpha \in \R$ such that $\alpha \le G_{\mu_a \boxplus \mu_b}(x),$ we have 
 \[ \lim_{\delta \downarrow 0} \lim_{N \rightarrow \infty}\frac{1}{N}\log P^{\otimes N}\left( v_N(x) \in \left[\alpha-\delta, \alpha+\delta\right]\right)
 =- T_{x,\ell}^-\left(\alpha\right).
 \]
\end{proposition}

We will not give a full proof of Proposition \ref{prop:16}. This follows from an adaptation of Lemma 18 and Proposition 16 in \citep{GuMa05}. In Lemma 18 in particular, one can check that the deviations above the mean may involve not only the limiting empirical distribution but also the limiting largest particle, whereas the deviations below the mean may depend on the limiting smallest particle.

Ror $\gamma:= (\gamma_1, \ldots, \gamma_p),$ we now define  by recursion,  for any $1 \le i \le p,$
\[ L^{(i)}_{\mathbf \gamma}(x) :=  
\left\{\begin{array}{ll}
    \inf\limits_{y \le\mathsf l({\mu_a} \boxplus{\mu_b}) } \left\{ T_{x,y}^-\left(\frac{1}{\gamma_i}\right) + I^{1}_{\rm min}(y)\right\}    , & \textrm{if } \mathsf r({\mu_a} \boxplus{\mu_b}) \le x \le K_{\mu_a \boxplus \mu_b}\left(\frac{1}{\gamma_i}\right),\\
     \inf\limits_{\mathsf r({\mu_a} \boxplus{\mu_b}) \le y \le x} \left\{ T_{x,y}^+\left(\frac {1}{\gamma_i}\right)+ L^{(i-1)}_\gamma(y)\right\}    , & \textrm{if } x  \ge K_{\mu_a \boxplus \mu_b}\left(\frac{1}{\gamma_i}\right),\\
    \infty, & \textrm{if } x <\mathsf r({\mu_a} \boxplus{\mu_b}),
                     \end{array}
\right.
\]
with the convention that 
\[
L^{(0)}_\gamma(y) :=     I^{1}(y), \quad\textrm{if } y \ge \mathsf{r}(\mu_a \boxplus \mu_b)
\]
and 
\[ 
K_{\mu_a \boxplus \mu_b}\left(\frac{1}{\gamma_i}\right) = \mathsf r({\mu_a} \boxplus{\mu_b})\quad \textrm{ if } G_{{\mu_a} \boxplus{\mu_b}}( \mathsf r({\mu_a} \boxplus{\mu_b})) \le \frac{1}{\gamma_i} 
\]

We can now state our main result
\begin{theorem}
\label{theo:deformed}
Under the assumptions \ref{hyp:esd}, \eqref{hyp:colle} and \ref{hyp:min}, for any $p \in \mathbb N^*$ and any $\gamma \in (\R_+)^p,$ the law of the largest eigenvalue $\widetilde{\la_{\rm max}^N}$ of the matrix $X_N$ defined in \eqref{def:deformed}
under $(m_N^1)^{\otimes (p+1)}$ satisfies a large deviation principle in the scale $N$ with good rate function~$L^{(p)}_\gamma.$
\end{theorem}

The rest of this section is devoted to the proof of Theorem \ref{theo:deformed} in the case $p=1.$  For $p>1,$ the proof is very similar, except that instead of 
conditioning by the deviations of the extreme eigenvalues of $H_N,$ we will condition of the deviations of extreme eigenvalues of the model at step  $p-1$.

\begin{proof}[Proof of Theorem \ref{theo:deformed} in the case $p=1$]
As in the proof of Theorem \ref{theo:main}, the exponential tightness is straightforward : for any $N \ge 1,$
\[ (m_N^1)^{\otimes 2}(\widetilde{\la_{\rm max}^N} \ge 2K+\gamma_1 + 1) = 0.
\]
We now prove  a weak large deviation principle.
 For $\gamma_1 >0,$ for any $z$ which does not belong to the spectrum of $H_N,$ one can write 
\[ {\rm det}(zI_N - X_N) =  {\rm det}(zI_N - H_N) \gamma_1 \left(\frac{1}{\gamma_1}- (U_1^{(1)})^*(zI_N-H_N)^{-1}U_1^{(1)}\right).
\]
Therefore, $z$ is an eigenvalue of $X_N$ which is not an eigenvalue of $H_N$ if and only if 
\[  (U_1^{(1)})^*(zI_N-H_N)^{-1}U_1^{(1)} = \frac{1}{\gamma_1}.
\]
By invariance by unitary conjugation, one can always assume that $H_N$ is diagonal, so that the latter reads
\[  \sum_{i=1}^N \frac{1}{z-\la_i^{(H_N)}}v_i^2 = \frac{1}{\gamma_1},
\]
where $v_i^2 = \frac{g_i^2}{ \frac{1}{N} \sum\limits_{i=1}^N g_i^2},$ with $(g_1, \ldots, g_N)$ having distribution $P^{\otimes N}.$

For any $(\la_1, \ldots, \la_N)$ fixed, the function 
\[ f_\la: z \mapsto  \frac{1}{N} \sum_{i=1}^N\frac{1}{z-\la_i} v_i^2
\] is decreasing and continuous, on $(\max_{i=1}^N \la_i, \infty),$ uniformly on $(v_1, \ldots, v_N)$ such that $ \sum_{i=1}^N v_i^2 = 1.$
Therefore, $f_\la(\widetilde{\la_{\rm max}^N}) = \frac{1}{\gamma_1},$ if and only if there exists a function $\varepsilon_\la$ going to zero at zero, such that for any $\delta >0$ small enough,
for any $x \in [ \widetilde{\la_{\rm max}^N}-\delta, \widetilde{\la_{\rm max}^N}+\delta],$
$f_\la(x) \in \left[ \frac{1}{\gamma_1}- \varepsilon_\la(\delta), \frac{1}{\gamma_1} +  \varepsilon_\la(\delta)\right].$
%with $ \varepsilon_\la(\delta)$ going to zero as $\delta$ goes to zero.
If we assume that $\eta,\delta < \frac{|x-y|}{4}$ and for all $i\in \mathbb N^*,$ $\la_i \le y+\eta,$ one can choose $\varepsilon_\la$ uniformly in $(\la_1, \ldots, \la_N).$
Moreover, if we denote by $\widetilde{v_N}(x):= \frac{1}{x-y}v_1^2 +  \sum_{i=2}^N \frac{1}{x-\la_i^{(H_N)}}v_i^2,$ we have the following: for any $\mathsf{r}(\mu_a \boxplus \mu_b)\le y <x,$ there exists a function $\varepsilon$ going to zero at zero such that,  for $\eta < \frac{|x-y|}{4}$ and $\delta$ small enough,
\begin{align*}
(m_N^1)^{\otimes 2}(\widetilde{\la_{\rm max}^N} \in [x-\delta, x+\delta]) & \ge  (m_N^1)^{\otimes 2}(\widetilde{\la_{\rm max}^N} \in [x-\delta, x+\delta] \cap \mathsf E_{N,\eta}^y) \\
& \ge  (m_N^1)^{\otimes 2}(\widetilde{\la_{\rm max}^N} \in [x-\delta, x+\delta] | \mathsf E_{N,\delta}^y) m_N^1(\mathsf E_{N,\eta}^y)\\
& \ge  (m_N^1)^{\otimes 2}\left(\widetilde{v^N}(x) \in \left[ \frac{1}{\gamma_1}- \varepsilon(\delta), \frac{1}{\gamma_1} +  \varepsilon(\delta)\right] | \mathsf E_{N,\delta}^y\right)\\
& \hspace*{6.5cm} \times m_N^1(\mathsf E_{N,\eta}^y), 
\end{align*}
where $\mathsf E_{N,\eta}^y$ was defined in \eqref{def:event}.

Assume that $ G_{\mu_a \boxplus \mu_b}(x) \le \frac{1}{\gamma_1}.$
By Proposition \ref{prop:16}, 
\[
\lim_{\delta \downarrow 0} \lim_{N \rightarrow \infty} \frac{1}{N}\log P^{\otimes N}\left( \widetilde{v_N}(x) \in \left[\frac{1}{\gamma_1}-\varepsilon(\delta), \frac{1}{\gamma_1}+\varepsilon(\delta)\right]| \mathsf E_{N,\eta}^y\right)
 =- T_{x,y}^+\left(\frac{1}{\gamma_i}\right),
\]
so that 
\[
\lim_{\delta \downarrow 0} \lim_{N \rightarrow \infty} \frac{1}{N}\log (m_N^1)^{\otimes 2}(\widetilde{\la_{\rm max}^N} \in [x-\delta, x+\delta]) 
\ge - T_{x,y}^+\left(\frac{1}{\gamma_1}\right) + \lim_{N \rightarrow \infty} \frac{1}{N}\log  m_N^1(\mathsf E_{N,\eta}^y).
\]
Taking the limit of the right hand-side as $\eta$ goes to zero, we get using Theorem \ref{theo:main} that
\[
\lim_{\delta \downarrow 0} \lim_{N \rightarrow \infty} \frac{1}{N}\log (m_N^1)^{\otimes 2}(\widetilde{\la_{\rm max}^N} \in [x-\delta, x+\delta]) 
\ge - T_{x,y}^+\left(\frac{1}{\gamma_1}\right) - I^1(y) \ge - L^{(1)}_\gamma(x),
\] 
where the last inequality was obtained by optimizing on $y.$

Assume now that  $ \mathsf{r}(\mu_a \boxplus \mu_b) < x < K_{\mu_a \boxplus \mu_b}\left(\frac{1}{\gamma_1}\right).$
We denote by   $ \mathsf{r}:=   \mathsf{r}(\mu_a \boxplus \mu_b),$ we define, similarly to \eqref{def:event}, for $y \le  \mathsf{l}(\mu_a \boxplus \mu_b)$
\[ \mathsf E_{N,\eta}^{y,-}:= \left\{ \la_{\rm min}^N \in [y-\eta, y+\eta] , \la_{\rm max}^N \in [\mathsf{r}-\eta, \mathsf{r}+\eta],\dd(\hat \mu_N, \nu_N^1) \le N^{-1/4}\right\},
\]
and we change the definition of $\widetilde{v_N}(x):=   \sum_{i=1}^{N-1} \frac{1}{x-\la_i^{(H_N)}}v_i^2 + \frac{1}{x-y}v_N^2.$ 
We can then write
\begin{align*}
(m_N^1)^{\otimes 2}(\widetilde{\la_{\rm max}^N} \in [x-\delta, x+\delta]) & \ge  (m_N^1)^{\otimes 2}(\widetilde{\la_{\rm max}^N} \in [x-\delta, x+\delta] \cap \mathsf E_{N,\eta}^{y,-}) \\
& \ge  (m_N^1)^{\otimes 2}(\widetilde{\la_{\rm max}^N} \in [x-\delta, x+\delta] | E_{N,\eta}^{y,-}) m_N^1(\mathsf E_{N,\eta}^{y,-})\\
& \ge  (m_N^1)^{\otimes 2}\left(\widetilde{v^N}(x) \in \left[ \frac{1}{\gamma_1}- \varepsilon(\delta), \frac{1}{\gamma_1} +  \varepsilon(\delta)\right] | E_{N,\eta}^{y,z,-}\right)\\
& \hspace*{6.5cm} \times m_N^1(E_{N,\eta}^{y,-}).
\end{align*}

In this case, 
by Proposition \ref{prop:16}, 
\[
\lim_{\delta \downarrow 0} \lim_{N \rightarrow \infty} \frac{1}{N}\log P^{\otimes N}\left( \widetilde{v_N}(x) \in \left[\frac{1}{\gamma_i}-\varepsilon(\delta), \frac{1}{\gamma_i}+\varepsilon(\delta)\right]| \mathsf E_{N,\eta}^{y,-}\right)
 =- T_{x,y}^-\left(\frac{1}{\gamma_1}\right),
\]
so that 
\begin{equation}\label{lbdeformed}
\lim_{\delta \downarrow 0} \lim_{N \rightarrow \infty} \frac{1}{N}\log (m_N^1)^{\otimes 2}(\widetilde{\la_{\rm max}^N} \in [x-\delta, x+\delta]) 
\ge - T_{x,y}^-\left(\frac{1}{\gamma_1}\right) + \lim_{N \rightarrow \infty} \frac{1}{N}\log  m_N^1(E_{N,\eta}^{y,-}).
\end{equation}

The last step to prove the lower bound in this case is to check 
\begin{equation}\label{indepminmax}
 \lim_{\eta \downarrow 0} \lim_{N \rightarrow \infty} \frac{1}{N}\log  m_N^1(E_{N,\eta}^{y,-}) \ge - I_{\rm min}(y).
\end{equation}
Then, taking the limit as $\eta$ goes to zero in \eqref{lbdeformed} and optimizing in $y$ gives the required lower bound.\\

We now prove \eqref{indepminmax}. Similarly to  Lemma \ref{uniqueness} and \ref{tiltedLGN} (by symmetry between the smallest and largest eigenvalue), one can show that there exists a unique $\theta_y \le 0$ such that,
for any $\eta >0$ and $N$ large enough,
\[
 m_N^{1, \theta_y}\left( \la_{\rm min}^N \in [y-\eta, y+\eta] , \dd(\hat \mu_N, \nu_N^1) \le N^{-1/4}\right) \ge \frac{2}{3}.
\]

One can also check that, for any $\theta_y \le 0$ and for any $\eta >0$ and  $N$ large enough,
\eq \label{typique}
 m_N^{1, \theta_y}(\la_{\rm max}^N \in [\mathsf{r}-\eta, \mathsf{r}+\eta]) \ge \frac{2}{3},
\qe
so that, for any $\eta >0$ and  $N$ large enough, 
\[ 
m_N^{1, \theta_y}(\mathsf E_{N,\eta}^{y,-}) \ge \frac{1}{3}.
\]
Indeed, \eqref{typique} comes from the following remark: if we set $\varphi(\theta):=m_N^{1, \theta}(\la_{\rm max}^N \ge \mathsf{r}+\eta),$ the function $\varphi$ is convex so that its derivative is increasing. At $\theta=0$, $\varphi$ and its  derivative go  exponentially fast to zero by the previous large deviation upper bound. Hence, for $\theta\le 0$ $\varphi$ goes exponentially fast to zero.

With this ingredient, the proof of \eqref{indepminmax} goes as in the proof of Proposition \ref{prop:down}:
\begin{align*}
 m_N^1(\mathsf E_{N,\eta}^{y,-}) &  = \dE_{m_N^1} \left( \mathsf{1}_{\mathsf E_{N,\eta}^{y,-}} \frac{I_N^1(\theta_y, H)}{I_N^1(\theta_y,H)}\right) \\
 & \ge  \inf_{U\in \mathsf E_{N,\eta}^{y,-}} \frac{1}{I_N^1(\theta_y, A+UBU^*)} 
   I_N^1(\theta_y, A)  I_N^1(\theta_y, B)m_N^{1,\theta_y}(\mathsf E_{N,\eta}^{y,-}),
\end{align*} 
so that, using again Proposition 2.1 in \citep{Ma07}, we get: 
\begin{align*}
 \lim_{\eta \downarrow 0}  \liminf_{N \rightarrow \infty} \frac{1}{N}\log m_N^1\left( \mathsf E_{N,\eta}^{y,-}\right) &\ge - I_{\rm min}(\theta_y,y) -  \lim_{\eta \downarrow 0}  g_{\theta_y}(\eta) = -  I_{\rm min}(y).\\
\end{align*}

The strategy to get the upper bound is similar : we know that, for $N$ large enough, $\la_1^{(H_N)} \in [\mathsf{r}(\mu_a \boxplus \mu_b),  \mathsf{r}(\mu_a)
+\mathsf{r}(\mu_b)+1]$ almost surely, so for any $\delta >0,$ there exists $p \in \mathbb N^*$ and $\rho_1, \ldots, \rho_p$ such that 
\[ m_N^1\left(\la_{\rm max}^N \in \cup_{i=1}^p [\rho_i-\delta, \rho_i+\delta]\right) =1.  
\]
Similarly, for any $\delta >0,$ there exists $\ell_1, \ldots, \ell_p$ such that 
\[ m_N^1\left(\la_{\rm min}^N \in \cup_{i=1}^p [\ell_i-\delta, \ell_i+\delta]\right) =1.  
\]

Assume that $ G_{\mu_a \boxplus \mu_b}(x) \le \frac{1}{\gamma_1}.$
\begin{align*}
(m_N^1)^{\otimes 2}(\widetilde{\la_{\rm max}^N} \in [x-\delta, x+\delta]) & \le (m_N^1)^{\otimes 2}(\widetilde{\la_{\rm max}^N} \in [x-\delta, x+\delta] \cap \{\dd(\hat \mu_N, \nu_N^1) \le N^{-1/4}\})\\
& + m_N^1(\dd(\hat \mu_N, \nu_N^1) > N^{-1/4}) \\
& \le  \sum_{i=1}^p (m_N^1)^{\otimes 2}(\widetilde{\la_{\rm max}^N} \in [x-\delta, x+\delta] \cap \mathsf E_{N,\delta}^{\rho_i})\\
&+ m_N^1(\dd(\hat \mu_N, \nu_N^1) > N^{-1/4})\\
& \le  \sum_{i=1}^p (m_N^1)^{\otimes 2}(\widetilde{\la_{\rm max}^N} \in [x-\delta, x+\delta] | \mathsf E_{N,\delta}^{\rho_i}) m_N^1(\mathsf E_{N,\delta}^{\rho_i}) \\
& + m_N^1(\dd(\hat \mu_N, \nu_N^1) > N^{-1/4})\\
\end{align*}
We then use Lemma \ref{lem:conc} to get rid of the last term and then let $\delta$ go to zero.\\

Assume now that $ G_{\mu_a \boxplus \mu_b}(x) \le \frac{1}{\gamma_1}.$ We apply the very same strategy with $\mathsf E_{N,\delta}^{\ell_i, -}$ instead of $\mathsf E_{N,\delta}^{\rho_i}$ and use the same ingredient together with the bound:
\[ \mathsf E_{N,\delta}^{\ell_i, -} \subset \{\la_{\rm min}^N \in \cup_{i=1}^p [\ell_i-\delta, \ell_i+\delta]\} \cap \{\dd(\hat \mu_N, \nu_N^1) \le N^{-1/4}\}.
\]

\end{proof}

\small%  
  \setlength{\bibsep}{.5em}
  \bibliographystyle{abbrvnat}%
  \bibliography{LDP}

\end{document}